\documentclass[11pt]{amsart}

\usepackage{amscd}
\usepackage{amsmath, amssymb,color, stmaryrd}
\usepackage{amsfonts}
\usepackage{comment}
\usepackage{enumerate}

\newcommand{\p}{\partial}
\newcommand{\pbar}{\ov{\partial}}

\newcommand{\abs}[1]{\left\lvert#1\right\rvert}

\newcommand{\ov}[1]{\overline{#1}}
\newcommand{\ul}[1]{\underline{#1}}

\newcommand{\ti}[1]{\widetilde{#1}}
\newcommand{\vp}{\varphi}
\newcommand{\e}{\varepsilon}

\newcommand{\mSH}{m\mathrm{SH}}

\newcommand{\Hm}{\mathrm{H}^m}
\newcommand{\MA}{\mathrm{MA}}
\newcommand{\Supp}{\mathrm{Supp}}

\newcommand{\R}{\mathbb{R}}
\newcommand{\C}{\mathbb{C}}

\renewcommand{\leq}{\leqslant}
\renewcommand{\geq}{\geqslant}

\newcommand{\be}{\begin{equation}}
\newcommand{\ee}{\end{equation}}

\newcommand{\PSH}{\textrm{PSH}}
\newcommand{\loc}{\textrm{loc}}

\begin{document}
\newtheorem{claim}{Claim}
\newtheorem{theorem}{Theorem}[section]
\newtheorem{lemma}[theorem]{Lemma}
\newtheorem{corollary}[theorem]{Corollary}
\newtheorem{proposition}[theorem]{Proposition}
\newtheorem{conjecture}[theorem]{Conjecture}
\newtheorem{example}[theorem]{Example}
\newtheorem{question}{Question}[section]
\theoremstyle{definition}
\newtheorem{definition}[theorem]{Definition}
\newtheorem{remark}[theorem]{Remark}

\numberwithin{equation}{section}

\newenvironment{spm}
    {\left( \begin{smallmatrix}
    }
    { 
     \end{smallmatrix} \right)
    }
    
\newenvironment{sbm}
    {\left[ \begin{smallmatrix}
    }
    { 
     \end{smallmatrix} \right]
    }

\title{Asymptotics for some Singular Monge-Amp\`{e}re Equations}

\begin{abstract}
Given a psh function $\vp\in\mathcal{E}(\Omega)$ and a smooth, bounded $\theta\geq 0$, it is known that one can solve the Monge-Amp\`{e}re equation $\MA(\vp_\theta)=\theta^n\MA(\vp)$, with some form of Dirichlet boundary values, by work of \AA hag--Cegrell--Czy\.{z}--Hi\d{\^{e}}p. Under some natural conditions, we show that $\vp_\theta$ is comparable to $\theta\vp$ on much of $\Omega$; especially, it is bounded on the interior of $\{\theta = 0\}$. Our results also apply to complex Hessian equations, and can be used to produce interesting Green's functions.
\end{abstract}

\author[N. McCleerey]{Nicholas McCleerey}
\address[N. McCleerey]{Dept of Mathematics, Purdue University, West Lafayette, Indiana 47907-2067, USA}
\email{nmccleer@purdue.edu}


\maketitle


\section{Introduction}\label{Introduction}


Suppose that $(\Omega^n, \omega)$ is a strongly $m$-pseudoconvex ($1\leq m\leq n$) K\"ahler manifold with (smooth, non-empty) boundary; we write $\mSH(\Omega)$ for the space of $m$-subharmonic ($m$-sh) functions on $\Omega$. We adopt the convention $0\cdot \infty = 0$. When $m = n$, then $\mSH(\Omega) = \PSH(\Omega)$, the space of psh function.

For $p\leq 0$, define the Riesz kernel $K_p: [0, \infty)\rightarrow \R\cup\{-\infty\}$ and its inverse, $L_p: [-\infty, 0)\rightarrow [0,\infty)$ to be:
\[
K_p(s) := \begin{cases} \log s &\text{ if }p = 0\\ \frac{1}{p} s^p &\text{ otherwise} \end{cases}\quad\text{ and }\quad L_p(t) := \begin{cases} e^t &\text{ if }p = 0\\ (p t)^{1/p} &\text{ otherwise.} \end{cases}
\]
Note that $L_p(t)$ is convex and increasing.

Our main result is the following (see Section 2 for notation):

\begin{theorem}\label{Theorem 1}
Suppose that $\psi\in \mathcal{E}_m(\Omega)$, the Cegrell class on $\Omega$, is such that $\psi\leq -1$ and $\Hm(\psi)$ is supported on $S_\psi := \{\psi = -\infty\}$; moreover, assume that $S_\psi$ is closed and that $\psi$ is smooth on $\Omega\setminus S_\psi$.

Additionally, suppose there exists $p \leq 0 < \delta$ such that $h := L_p(\psi)$ satisfies
\begin{equation}\label{delta-regular}
i\p\pbar h \geq_m \e h^{1-\delta}\omega\text{ for some $\e > 0$.}
\end{equation}
Then for any $\theta\in C^2(\Omega)$, $0\leq \theta\leq 1$, the envelope:
\[
\psi_\theta := \sup\{ \vp\in\mSH(\Omega)\ |\ \vp\leq \min\{0, \theta\psi + C\}\text{ for some } C\geq 0\}^*,
\]
satisfies:
\begin{enumerate}[(i)]
\item $\Hm(\psi_\theta) = \chi_{S_\psi}\theta^m\Hm(\psi)$,
\item $\psi_\theta \geq \psi - C$,
\item $\psi_\theta \in L^\infty_\loc(\Omega\setminus \ov{\{\theta > 0\}})$, and
\item $\displaystyle \lim_{\substack{x\rightarrow z \\ x\not\in S_\psi}} \frac{\psi_\theta}{\psi} = \theta$ for all $z\in S_\psi\cap \{\theta > 0\}$.
\end{enumerate}
Moreover, if $p > -\frac{\delta}{2}$, then $\psi_\theta$ and $\theta\psi$ have the same singularity type.

\end{theorem}

Our main contribution is the construction of sub- and supersolutions to deduce the asymptotics in $(iii)$ and $(iv)$, which are new even in the psh case. In \cite{ACCH09, HP17}, envelopes were constructed which satisfy $(i)$ and $(ii)$ -- in fact, their construction works more generally, for any $\psi\in\mathcal{E}_m(\Omega)$ and any $\theta \in L^\infty(\Omega, \Hm(\psi))$. It should be the case that their envelope agrees with $\psi_\theta$ when the boundary values do, as a result of the general comparison theorem of \AA hag-Czy\.{z}-Lu-Rashkovskii \cite{ACLR24b}, although we did not check this thoroughly.

The relevance of Theorem \ref{Theorem 1} stems from the following: a widely held expectation in pluripotential theory is that there should not exist a psh function $\psi\in \mathcal{E}(\Omega)$ such that $\MA(\psi) = \delta_V$, the Hausdorff measure of a complex submanifold $V\subset\Omega$ (with non-zero dimension). If one could establish boundedness of the original envelopes in \cite{ACCH09} away from the support of $\theta$, then the expectation would follow from a straightforward contradiction argument. Theorem \ref{Theorem 1} provides exactly this, but under some additional assumptions; as such, it seems better to regard it simply as further evidence towards the expectation.

Let us then briefly discuss the necessity of our assumptions. Weakening the smoothness of $\psi$ in Theorem \ref{Theorem 1} to local boundedness should be quite doable -- we have not completely checked the details however, and hope to update this in the future. Going further than this would require a new method or finer computations; the same is true for the closedness of $S_\psi$.

The assumption that $\Hm(\psi)$ be supported on $S_\psi$ can be weakened substantially -- in fact, we do not even need to assume that $\psi$ is $m$-sh, only that $L_p(\psi)$ is. The specific assumptions and results are given in Theorem \ref{Real Theorem 1}.

Some form of assumption \eqref{delta-regular} seems indispensable however; given any $\psi\in\mSH(\Omega)$ (e.g. $\psi = \log\abs{z_1}$) one can check that for $h := L_p(\vp + \rho)$, $p\leq 0$, we have
\[
i\p\pbar h \geq_m \e h^{1-p} \omega
\]
for $\e > 0$ sufficiently small. Thus, Theorem \ref{Theorem 1} (really, Theorem \ref{Real Theorem 1}) is false without assuming something to replace \eqref{delta-regular}.

To understand \eqref{delta-regular}, we offer the following heuristic argument. Suppose that $p=0$ and $m=n$. Then
\[
i\p\pbar e^\psi = e^\psi i\p\pbar \psi + e^\psi i\p\psi \wedge \pbar\psi \geq \e e^{(1-\delta)\psi}\omega.
\]
Since $\MA(\psi) = 0$ on $\{\psi > -\infty\}$, $i\p\pbar\psi$ cannot have full rank -- if we weaken this condition, as in Theorem \ref{Real Theorem 1}, then the example above tells us we should refine this to saying that $i\p\pbar\psi$ cannot have full rank ``asymptotically." Assumption \eqref{delta-regular} ensures however that it has rank $n-1$ asymptotically, since the gradient term has rank 1. Varying $p$ changes the overall scaling on the right-hand side, so that larger $p$ corresponds to more asymptotic positivity; similarly with $\delta$. We thus view assumption \eqref{delta-regular} as a kind of ``maximal rank" condition.

\medskip

As $m$ decreases, it becomes easier to sastisfy condition \eqref{delta-regular}, since $i\p\pbar\psi$ needs less positivity to dominate $\omega$ in the $m$-sh sense (it only needs $m$-many positive eigenvalues). This, plus the additionally flexibility contained in Theorem \ref{Real Theorem 1}, allows us to show:
\begin{corollary}\label{Corollary 1}
Suppose that $V\subset \Omega$ is a complex submanifold, with $\p V =\emptyset$ or $\p V\subset \p\Omega$, and complex codimension $k \geq m$. Set $h := d_{\omega, V}^2$, the square of the Riemannian distance function to $V$ induced by $\omega$, and $p := 1 - \frac{k}{m}\leq 0$. Then for any $\theta\in C^2(\ov{\Omega})$, $\theta\geq 0$, the envelope:
\[
\psi_\theta := \sup\{ \vp\in\mSH(\Omega)\ |\ \vp\leq \min\{0, \theta K_p(h) + C\}\text{ for some } C\geq 0\}^*
\]
satisfies conclusions $(i)$-$(iv)$ in Theorem \ref{Real Theorem 1}. Moreover, if $k < \frac{3}{2}m$, then $\psi_\theta$ and $\theta K_p(h)$ have the same singularity type. 
%
%
\end{corollary}

The ``threshold value" $m = \lfloor \frac{n}{2}\rfloor$ is common in the study of $m$-convex functions, so it seems reasonable to expect that $\psi_\theta$ and $\theta K_p(h)$ should have the same singularity type when $k< 2m$, instead of only when $k < \frac{3}{2}m$. We can show this under some more restrictive hypothesis on $\theta$ (essentially, if it is the pullback of a function on $V$); this allows us to effectively use the twisting of the eigenspaces of $i\p\pbar(\theta K_p(h))$ to pull out an extra factor of $h$ in the estimates. We plan to include details at a future date.

\medskip

A few words about the proof of Theorem \ref{Theorem 1}. The construction of the subsolutions simultaneously generalizes and simplifies a result from \cite{CM24}; there, it was only carried out in the setting of Corollary \ref{Corollary 1}. Our construction (and indeed, our definition) of supersolutions is novel -- we work exclusively with respect to the linearization of $\Hm$ at $\psi$, instead of with $\Hm$ directly. This allows us to construct supersolutions which are very poorly behaved on $S_\psi$, but whose asymptotics are clear. This is also why we are able to work with $\psi$ which are not necessarily $m$-sh.

We conclude with some comments on further consequences and directions. It follows more-or-less immediately from Theorem \ref{Theorem 1} that one can apply Demailly's Lelong-Jensen formula (see \cite{Dem_CADG_book}) to the envelope $\psi_\theta$ when $\theta$ has compact support; this provides, for the $m$-sh functions especially, a rich class of new sub-mean values inequalities. In order to apply these effectively, one needs a better understanding of the level sets of $\psi_\theta$ -- this is especially interesting for the functions coming from Corollary \ref{Corollary 1}. Understanding if Poisson kernels exist for these functions, c.f. \cite{BST21}, would also be interesting.

\medskip

The rest of this paper is organized as follows. In Section 2, we collect some background results and fix notation. In Section 3, we prove Theorem \ref{Real Theorem 1}, which implies Theorem \ref{Theorem 1}. Corollary \ref{Corollary 1} is shown in subsection \ref{Corollary_Section}.

\subsection*{Acknowledgments} We would like to thank L\'{a}szl\'{o} Lempert and Norm Levenberg for interesting discussions.


\section{Background}\label{Background}


We collect several definitions and background results we will need. From now on, we let $(\Omega^n, \omega)$ be a strongly $m$-pseudoconvex K\"ahler manifold, with smooth, non-empty boundary and fixed defining function $\rho$ (Def. \ref{mpcnx}); additionally, we assume that $\Omega$ is compactly contained in a larger K\"ahler manifold, and that $\omega$ extends smoothly as a K\"ahler metric to the larger space.

\subsection{Basics and Notation}

Fix an integer $1\leq m \leq n$. 
\begin{definition}
We say that a $(1,1)$-form $\alpha$ is {\bf $m$-positive} if $\alpha^k\wedge\omega^{n-k}\geq 0$ for all $1\leq k\leq m$. Similarly, we say $u\in C^2(\Omega)$ is {\bf $m$-subharmonic} if $i\p\pbar u$ is $m$-positive.

Following B\l ocki \cite{Bl05}, we say that a $(1,1)$-current $T$ is {\bf $m$-positive} if $T\wedge\alpha_2\wedge\ldots \wedge\alpha_m\wedge\omega^{n-m}\geq 0$ for all $m$-positive $\alpha_2,\ldots, \alpha_m$; we say that an upper-semicontinuous function $\vp$ is {\bf $m$-subharmonic} (abbrev. {\bf $m$-sh}) if $i\p\pbar \vp$ is $m$-positive, and write $\vp \in \mSH(\Omega)$ in this case. 
\end{definition}

By G\aa rding's inequality \cite{Gaar59, DL15}, the two definitions are consistent for $u\in\mSH(\Omega) \cap  C^2(\Omega)$. 

The restriction of an $m$-sh function to a proper complex hyperplane of dimension $k$ will, in general, only be $(m-k)$-sh, provided that $k > n- m$. This is significantly weaker than the restriction property for psh functions; further, it does not characterize $m$-subharmonicity, again, unlike the psh case (e.g. $\abs{z_1}^2 + \e\abs{z_2}^2 - \e\abs{z_3}^2$ on $\C^3$ with the standard metric is not $2$-sh).

Given $(1,1)$-currents $S$ and $T$, we write $T\geq_m S$ to mean that $T - S$ is $m$-positive. We say that $T$ is {\bf strictly $m$-positive} if $T \geq_m \e\omega$ for some $\e > 0$. Similarly, $\vp\in\mSH(\Omega)$ is {\bf strictly $m$-sh} if $i\p\pbar \vp$ is strictly $m$-positive.

\begin{definition}\label{mpcnx}
We say that a complex manifold with boundary $\Omega$ is {\bf strongly $m$-pseudoconvex} if there exists a defining function $\rho\in \mSH(\Omega)\cap C^\infty(\ov{\Omega})$, strongly $m$-sh, and such that $\p\Omega = \{\rho = 0\}$. 
\end{definition}


For $u\in\mSH(\Omega)\cap C^2(\Omega)$, we define the {\bf complex $m$-Hessian measure}:
\[
\Hm(u) := (i\p\pbar u)^m\wedge\omega^{n-m}.
\]
It is a positive Radon measure on $\Omega$, which is weakly continuous under monotone convergence; the obvious multilinear extension enjoys the same properties.

The largest space of $m$-sh functions to which $\Hm$ extends as a continuous function under decreasing sequences is the {\bf Cegrell class} \cite{Ceg84}, $\mathcal{E}_m(\Omega)$, sometimes also known as the {\bf domain of definition} of $\Hm$.

For $\vp\in\mathcal{E}_m(\Omega)$, the measure $\Hm(\vp)$ exists, and can be decomposed into a $m$-polar part $\Hm_s(\vp)$, and a non-$m$-polar part, $\Hm_r(\vp)$:
\[
\Hm(\vp) = \Hm_r(\vp) + \Hm_s(\vp);
\]
see e.g. \cite{ACLR24b}. The polar part $\Hm_s(\vp)$ will be supported on the {\bf polar set} of $\vp$, $S_\vp := \{\vp = -\infty\}$. We write $\mathcal{E}^a_m(\Omega)$ for those $\vp\in\mathcal{E}_m(\Omega)$ with $\Hm_s(\vp) = 0$.

Unlike the non-polar part, the polar part can often be read off from pointwise bounds on $\vp$.
\begin{definition}
Suppose that $\vp, \psi : \Omega\rightarrow \R\cup\{\pm\infty\}$ are continuous at each point in $S_\vp, S_\psi$, respectively.\footnote{This holds if $\vp,\psi$ are upper-semicontinuous, but we do not assume this.} We say that:
\begin{enumerate}

\item $\vp$ is {\bf more singular} than $\psi$, and write $[\psi]\leq [\vp]$, if for each $K\Subset\Omega$, there exists a constant $C_K$ such that $\vp \leq \psi + C_K$. We say that $\vp$ and $\psi$ have the same {\bf singularity type}, and write $[\vp] = [\psi]$, if $[\vp]\leq [\psi]$  and $[\psi]\leq[\vp]$.

\item $\vp$ is {\bf asymptotically more singular} than $\psi$, and write $\psi\preccurlyeq\vp$, if
\[
\limsup_{\substack{x\rightarrow z \\ x\not\in S_\psi}} \frac{\psi(x)}{\vp(x)} \leq 1\text{ for all }z\in S_\psi,
\]
were we adopt the convention that $\frac{c}{-\infty} = 0$ for any $c \in \R$. We say that $\vp$ and $\psi$ are {\bf asymptotic}, and write $\vp\asymp\psi$, if $\vp\preccurlyeq\psi$ and $\psi \preccurlyeq\vp$.

\item $\vp$ is {\bf essentially more singular} than $\psi$ if, for each open $U\Subset\Omega$, there exists $w\in\mathcal{E}^a_m(U)$ such that $\vp + w \leq \psi$ on $U$.
%

\end{enumerate}
\end{definition}

If $[\psi]\leq[\vp]$, then $\psi\preccurlyeq \vp$; if $\psi \preccurlyeq \vp$, then $\vp$ is essentially more singular than $\psi$ (argue as in \cite[Lemma 5.3 (1)]{ACLR24b}, for instance). 
All three conditions imply that $\Hm_s(\psi) \leq \Hm_s(\vp)$, as we will now show. First, a well-known result of Demailly \cite{Dem93} (see also \cite[Lemma 4.1]{ACCH09}) implies the following:
\begin{proposition}\label{polar comparison}
Suppose that $\vp,\psi\in\mathcal{E}_m(\Omega)$ and that $\theta\geq 0$ is a Borel measurable function, defined everywhere, which is bounded from above. If:
\[
\limsup_{\substack{x\rightarrow z \\ x\not\in S_\psi}} \frac{\psi(x)}{\vp(x)} \leq \theta(z) \text{ for all } z\in S_\psi, \text{ then }\Hm_s(\psi) \leq \theta^m\Hm_s(\vp).
\]
\end{proposition}
\begin{proof}
Suppose that $E$ is a measurable set and $A \geq 0$ is such that $\theta \leq A$ on $E$. Fix $\e > 0$; then for any $z\in E\cap S_\psi$, there exists a ball $B(z)$ such that $\frac{\psi(x)}{\vp(x)}\leq \theta(z) + \e$ on $B(z)\setminus S_\psi$. It follows that $\psi \geq (\theta(z) + \e)\vp$ on $B(z)$, from basic properties of subharmonic functions, and so by \cite[Lemma 4.1]{ACCH09}, we have:
\[
\Hm_s(\psi) \leq (\theta(z)+\e)^m\Hm_s(\vp)\quad\quad\text{ on }B(z).
\]
Since this holds for any $z\in E\cap S_\psi$, we see that $\Hm_s(\psi) \leq (A+\e)^m\Hm_s(\vp)$ on $E$. Approximating $\theta$ by a decreasing sequence of simple functions, and subsequently letting $\e\rightarrow 0$, the result follows.
\end{proof}

\begin{proposition}
Suppose that $\vp, \psi\in\mathcal{E}_m(\Omega)$, and that $w\in\mathcal{E}^a_m(\Omega)$ is such that $\vp + w \leq \psi$. Then:
\[
\Hm_s(\psi) \leq \Hm_s(\vp).
\]
\end{proposition}
\begin{proof}
For any Borel $m$-polar set $A$, we have that:
\[
\int_A (i\p\pbar (\vp + w))^m\wedge\omega^{n-m} \leq \sum_{k=0}^m {m \choose k}\left(\int_A \Hm(\vp)\right)^{\frac{k}{m}}\left(\int_A \Hm(w)\right)^{\frac{m-k}{m}} = \int_A \Hm(\vp),
\]
by \cite[Lemma 4.4]{ACCH09} (although stated for the psh case in \cite{ACCH09}, the proof can be copied exactly; the key results from \cite{Ceg04} which are used (Lemma 5.4 and Theorem 5.5) can be proved in the same way in the $m$-sh setting).  Since $\vp + w \leq \psi$, it follows from Proposition \ref{polar comparison} that $\Hm_s(\psi) \leq \Hm_s(\vp + w) = \Hm_s(\vp)$.
\end{proof}

\subsection{Supersolutions and Polar Mass}

We need some further definitions to work with envelopes of $m$-sh functions. The notion of ``subsolution" in our context is basically just being a candidate in the given envelope. The notion of supersolution is more subtle. We have a preliminary definition:

\begin{definition}\label{m-elliptic_definition}
Suppose that $m\geq 2$ and that $T$ is an $(n-1, n-1)$-current on an open subset $U\subset \Omega$ such that:
\begin{itemize}
\item there exist $u_2, \ldots, u_m\in\mSH(U)\cap C^\infty(U)$ with $T = i\p\pbar u_2\wedge\ldots \wedge i\p\pbar u_m\wedge\omega^{n-m}$, and
\item $\omega \wedge T > 0$ on $U$.
\end{itemize}
We then define an {\bf $m$-elliptic operator} on $U$ by:
\[
\Delta g := i\p\pbar g \wedge T\quad \quad \text{ for any }g \in \mSH^\delta(U),
\]
where $\mSH^\delta(U) := \{u - v\ |\ u, v\in\mSH(U)\}$. Note that $\Delta g$ is a signed Radon measure on $U$.
\end{definition}

%

\begin{definition}\label{superweight def}

We say that $\psi: \Omega \rightarrow \R\cup \{\pm\infty\}$ is a {\bf superweight} if there exists a closed, complete $m$-polar set $E$ such that:
\begin{enumerate}
\item $\{\psi \leq -1\}\Subset \Omega$,
\item $\exists\, u, v\in \mSH(\Omega\setminus E)\cap L^\infty_\loc(\Omega\setminus E)$ such that $\psi = u - v$ on $\Omega\setminus E$,
\item $\psi$ is continuous at each point in $(\Omega\setminus E)\cup S_\psi$, and
\item there exists an $m$-elliptic operator $\Delta$ on $\Omega\setminus E$ such that $\Delta \psi \leq 0$.
\end{enumerate}

\end{definition}

\noindent Condition (2) implies that $S_\psi \subseteq E$, but note that $S_\psi$ need not be closed.

If $\psi$ is a smooth $m$-sh function and $\Delta$ is an $m$-elliptic operator such that $\Delta \psi \leq 0$ on an open set $U$, then the G\aa rding inequality implies that $\Hm(\psi) \equiv 0$ on $U$; for more general $\psi$ and $\Delta$, we expect the same result to hold, although it seems one cannot directly apply the generalized G\aa rding inequality of Dinew-Lu \cite[Theorem 1]{DL15}.

\subsection{Extensions}

Next, we have a basic extension result, which we will use repeatedly in the next section:
\begin{proposition}\label{extension}

Suppose that $E\subset \Omega$ is a closed $m$-polar set. If $\vp\in \mSH(\Omega\setminus E)$, such that $\vp \leq 0$, then $\vp$ admits a unique extension across $E$, also denoted $\vp$, such that $\vp\in\mSH(\Omega)$.

\end{proposition}
\begin{proof}

Let $\psi\in\mSH(\Omega)$ be such that $E \subseteq S_{\psi}$ and $\psi \leq 0$  \cite{GZ_book}; suppose also that $\alpha_2, \ldots, \alpha_m$ are smooth $m$-positive $(1,1)$-forms on $\Omega$. Let $\Delta_\alpha := \alpha_2\wedge\ldots\alpha_m\wedge\omega^{n-m}$. 

For any $\delta > 0$, we have $\vp + \delta\psi \in \mSH(\Omega\setminus E)$; by definition, $\vp + \delta\psi$ is $\Delta_\alpha$-subharmonic on $\Omega\setminus E$. Extend $\vp + \delta\psi$ by setting $\vp + \delta\psi \equiv -\infty$ on $E$ and write $\ti{\vp}_{\delta}$ for the extension. It will suffice to show that $\ti{\vp}_{\delta}$ is $\Delta_\alpha$-subharmonic. 

Consider the canonical cut-offs $\ti{\vp}_{\delta, s}:=\max\{\ti{\vp}_\delta, s\}$, $s \leq 0$. Since $\{\vp + \delta \psi < -s\}\supseteq \{\delta \psi < -s\}$, we have $\ti{\vp}_{\delta, s} \equiv s$ on a neighborhood of $E$; it follows that $\ti{\vp}_{\delta, s}$ is $\Delta_{\alpha}$-subharmonic on $\Omega$. Letting $s\rightarrow -\infty$, we conclude that $\ti{\vp}_\delta$ is a limit of $\Delta_\alpha$-subharmonic functions, and hence is $\Delta_\alpha$-subharmonic itself.
\end{proof}

\subsection{Three-Circle's Theorem and Corollaries}

We will also need a corollary to (a version of) the three-circles theorem for supersolutions. 

\begin{proposition}\label{three circles}
Suppose that $\psi$ is a supersolution. Then for any $\vp\in\mSH(\Omega)$, $\sup_\Omega \vp = -1$, the function $s\mapsto \max_{\{\psi \leq s\}} \vp$ is convex and increasing on $(-\infty, -1)$.
\end{proposition}
\begin{proof}
Let $\Delta$ and $E$ be as in Definition \ref{superweight def}. By Hartog's Lemma \cite[Theorem 1.46]{GZ_book}, it suffices to assume that $\vp$ is smooth. 
Let $u\leq -1$ be an $m$-sh function on $\Omega$ with $S_u = E$ 
and set $M_s := \max_{\{\psi \leq s\}} \vp$ for $s\in (-\infty, -1]$. Fix $-\infty < s_1 < s_2 \leq -1$ and consider the open set:
\[
W = \{s_1 < \psi < s_2\}.
\]
Note that the boundary of $W$ in $\Omega$ consists of (up to) three distinct parts -- the sets $\{\psi = s_1\}$, $\{\psi =s_2\}$, and some (potentially empty) subset of $E\setminus S_\psi$ where we do not know that values of $\psi$.

Define the function $F(z) = \frac{s_2 - {\psi}(z)}{s_2 - s_1}M_{s_1} + \frac{{\psi}(z) - s_1}{s_2 - s_1}M_{s_2} = \frac{M_{s_2} - M_{s_1}}{s_2 - s_1}\psi(z) + \frac{s_2M_{s_1} - s_1M_{s_2}}{s_2 - s_1}$ on $W$. This satisfies $\Delta(F) = \frac{M_{s_2} - M_{s_1}}{s_2 - s_1}\Delta({\psi}) \leq 0$; note that $F > M_{s_1}$ on $W$, but this may not hold on the part of the boundary of $W$ in $E\setminus S_\psi$.

Consider the $m$-sh function $\vp_\e := \vp  + \e(u + \rho)$, $\e > 0$, which is strictly $m$-sh and less than $\vp$. By continuity of $u$ at $-\infty$, we see that there is some neighborhood $U$ of $E$ in $\Omega$ such that $\vp_\e \leq M_{s_1} - 1 < F$ on $\ov{U}$; we have then $F\geq \vp_\e$ on the boundary $W\cap (\Omega\setminus \ov{U})$.

Since $\Delta \vp_\e \geq \e \Delta \rho > 0\geq \Delta F$, the maximum principle 
implies that $\vp_\e \leq F$ on $W\cap (\Omega\setminus \ov{U})$, and hence also on $W$. Taking the limit $\e\rightarrow 0$ and then the max over $\{\psi \leq s\}$ (for $s_1 < s < s_2$) we see that $M_s$ is convex; it is obviously increasing. 
\end{proof}

\begin{corollary}\label{slopes}

Suppose that $\psi$ is a supersolution and $\vp\in\mSH(\Omega)$, $\sup_\Omega \vp = -1$; set $M_s(\vp) := \max_{\{\psi\leq s\}} \vp$ for $s\in(-\infty, -1)$. Then
\[
\frac{M_s(\vp) - M_t(\vp)}{s - t}
\]
is increasing in $s$ and $t$ (separately), where $-\infty < s < t < -1$. It follows that:
\[
\sigma^*_\psi(\vp) = \lim_{s\searrow-\infty} \frac{M_s(\vp)}{s} = \max\{\gamma \geq 0\ |\ \vp \leq \gamma\psi + C\}.
\]
\end{corollary}

The following is due to Rashkovskii-Sigurdsson \cite[Lemma 4.1]{RS15}, and is the corollary that we will need later:

\begin{corollary}\label{RS_Cor}

Suppose that $\psi$ is a supersolution, and $\vp\in\mSH(\Omega)$ is such that $[\psi]\leq [\vp]$ and $\sup_\Omega \vp \leq -1$. Then $\vp \leq \psi$.

\end{corollary}
\begin{proof}

Since $[\psi]\leq [\vp]$, we have $\sigma^*_\psi(\vp) \geq 1$. Corollary \ref{slopes} then tells us that:
\[
1 \leq \sigma_{\psi}^*(\vp) \leq \frac{M_{s}(\vp) - M_{-1}(\vp)}{s + 1}\text{ for all }s\in (-\infty, -1).
\]
It follows that, for each $s\in(-\infty, -1)$, we have $\psi(z) = s \geq M_s(\vp) - M_{-1}(\vp) - 1 \geq \vp(z)$ for all $z\in \{\psi = s\}$. In particular, we conclude that $\psi(z)\geq \vp(z)$ for all $z\in \{\psi \leq -1\}$. We conclude, as $\psi \geq \vp$ on $\{\psi > -1\}$.
\end{proof}


\section{Main Results}\label{Main Results}


In this section, we prove our main results, Theorem \ref{Theorem 1} and Corollary \ref{Corollary 1}. Theorem \ref{Theorem 1} will follow from combining the subsolutions constructed in subsection \ref{Subsolution_Section} with the supersolutions constructed in subsection \ref{Supersolution_Section}; this is done in Theorem \ref{Real Theorem 1}. Corollary \ref{Corollary 1} is shown in subsection \ref{Corollary_Section}.

As mentioned in the Introduction, we will not work with $\psi$ directly, but rather with a bounded $m$-sh function $h$ which has an $m$-polar minimum set; if it happens that $K_p(h)\in \mathcal{E}_m(\Omega)$, then we (essentially) recover the set-up in the Introduction. Working with $h$ directly however is much more flexible.

For the rest of the section, we fix constants $\e, \delta > 0$ and assume that $h$ satisfies the following assumptions:
\begin{itemize}
\item $h\in \mSH(\Omega)$, $0\leq h\leq 1$,
\item $\{h = 0\}$ is non-empty and $m$-polar,
\item $h\in C^2(\{h > 0\})$, and
\item $i\p\pbar h \geq_m \e h^{1-\delta}\omega$ on $\{h > 0\}$.
\end{itemize}


\subsection{Subsolution}\label{Subsolution_Section}


We start with a lemma:
\begin{lemma}\label{p lemma}
Suppose that $p< 1$ is such that:
\begin{equation}\label{growth condition}
i\p\pbar K_p(h) \geq_m -ch^{\ti{\gamma}}K_p'(h)\omega = -ch^{p-1 + \ti{\gamma}}\omega\quad  \text{ on }\{ h > 0\},
\end{equation}
for some $c > 0$ and $\ti{\gamma} > 1- \delta$. Then for each $p < q \leq 1$, there exists some $C > 0$ such that $K_q(h) + C\rho\in \mSH(\Omega)$.
\end{lemma}
\begin{proof}
On $\{h > 0\}$, we have:
\begin{align}
i\p\pbar K_q(h) &= K_q'(h) i\p\pbar h + K''_q(h)i\p h\wedge\pbar h \notag\\
& = \frac{K_q''(h)}{K_p''(h)} i\p\pbar K_p(h) + \left(K'_q(h) - \frac{K_q''(h)}{K_p''(h)} K_p'(h)\right)i\p\pbar h \notag\\
&\geq_m -c\frac{1-q}{1-p}h^{\ti{\gamma}} K_q'(h)\omega + \e\frac{q-p}{1-p}K_q'(h)h^{1-\delta}\omega. \label{lower K bound}
\end{align}
We see that $K_q(h)$ is $m$-sh on $\{h^{\ti{\gamma} + \delta - 1} < \frac{\e (q-p)}{c(1-q)}\}$ -- by  Proposition \ref{extension}, $K_q(h) + C\rho$ is $m$-sh on $\Omega$ for $q > p$ and $C$ sufficiently large.
\end{proof}

\begin{proposition}\label{subsolution}
Fix $p \leq 0$, and suppose that $h, c,$ and $\ti{\gamma}$ are as in Lemma \ref{p lemma}. Fix $0 < \gamma < \min\{\ti{\gamma} +\delta - 1, \tfrac{\delta}{2}\}$, and set $\ell$ to be the smallest integer such that $p + \ell \gamma > 0$; define $q_j := p + j \gamma$ for each $j = 1, \ldots, \ell$.

Let $0 \leq \theta\leq 1$ and $0\leq \theta_j\leq 1$, $j = 1,\ldots, \ell$, be smooth functions with:
\begin{itemize}
\item $\theta_0 := \theta$ and $\theta_\ell = 1$, 
\item $\Supp(\theta_{j-1})\subset \{\theta_j = 1\}^\circ$ for each $j$.
\end{itemize}
Let also $a_j  >0$, $j = 1, \ldots, \ell$, be positive constants. Then:
\[
\theta K_p(h) +  \sum_{j=1}^\ell a_j \theta_j K_{q_j}(h) + C \rho\ \  \in \mSH(\Omega)
\]
for all $C$ sufficiently large.
\end{proposition}
\begin{proof}

If we define $q_0 := p$ and $a_0 := 1$, then by Lemma \ref{p lemma}, it will suffice to show that $a_{j-1} \theta_{j-1} K_{q_{j-1}}(h) + a_{j} K_{q_{j-1}+\gamma}(h) + C\rho$ is $m$-sh on $\{\theta_{j} = 1\}$; since the sum is finite, we can conclude by increasing $C$ at the end. We thus set $a_{j-1} = 1$ and drop the remaining indices, without loss of generality.

Recall that we are writing $L_q(t)$ for the inverse of $K_q(s)$ on $\{t < 0\}$. Define the convex increasing function $\chi(t) := K_{q+\gamma}(L_q(t))$; we have $\chi'(t) = L_q(t)^\gamma$, $\chi''(t) = \gamma L_q(t)^{\gamma - q}$, and $K_{q+\gamma}(h) = \chi(K_q(h))$.

If we choose $C$ sufficiently large, then on $\{h > 0\}$, by completing the square and using \eqref{lower K bound}, we have:
\begin{align*}
i\p\pbar (\theta K_q(h) + a K_{q+\gamma}(h)) &=  \theta i\p\pbar K_q + i\p K_q \wedge \pbar \theta + i\p\theta\wedge\pbar K_q + K_q i\p\pbar \theta \\
& + \frac{a}{2}i\p\pbar K_{q+\gamma} + \frac{a}{2}\chi'(K_q)i\p\pbar K_q + \frac{a}{2}\chi''(K_q)i\p K_q \wedge\pbar K_q\\
&\geq_m \frac{a}{2}i\p\pbar K_{q+\gamma} + \left(\theta + \frac{a}{2}\chi'(K_q)\right) \left(- c\frac{1-q}{1-p}h^{\ti{\gamma}} \right)K_q' \omega\\
& - \frac{4}{a \chi''(K_q)}i\p\theta\wedge\pbar\theta + K_q i\p\pbar \theta.
\end{align*}
From \eqref{lower K bound} again, we have $i\p\pbar K_{q+\gamma} \geq_m \e' K_{q+\gamma}'\, h^{1-\delta}\omega - C\omega$, for some small $\e' > 0$, so that:
\[
i\p\pbar (\theta K_q + a K_{q+\gamma}) \geq_m \left(\frac{a\e'}{2} h^{q + \gamma - \delta} - C h^{q - 1 + \ti{\gamma}} - C h^{q - \gamma} - C \abs{K_q(h)} - C\right)\omega.
\]
Since $\gamma < \min\{\ti{\gamma} - 1 + \delta, \frac{\delta}{2}\}$, we have that $h^{q + \gamma - \delta} > C\max\{h^{q - 1 + \ti{\gamma}}, h^{q-\gamma}\}$ when $h$ is sufficiently small, so that $\theta K_q + a K_{q+\gamma} + C\rho$ is $m$-sh on $\{h > 0\}$; we conclude by Proposition \ref{extension}.
\end{proof}

\begin{remark}\label{bounded_remark}
In Proposition \ref{subsolution}, if $p > \max\{1-\delta-\ti{\gamma}, -\frac{\delta}{2}\}$, then we can take $\gamma > 0$ such that $q_1 = p + \gamma > 0$, so that $K_{q_1}(h)$ is bounded.
\end{remark}


\subsection{Supersolution}\label{Supersolution_Section}


We now construct the supersolution we need; the construction is similar to the previous one.
\begin{proposition}\label{supersolution}
Suppose that $p\leq 0$ is such that:
\begin{itemize}
\item for each $p < q \leq 1$, there exists $C_q \geq 0$ such that  $K_q(h) + C_q\rho\in\mSH(\Omega)$, and
\item there exists an $m$-elliptic linear operator $\Delta$ such that:
\[
\ \ \quad\quad\quad\quad \Delta K_p(h) \leq ch^{\ti{\gamma}}K_p'(h)\Delta\rho\quad\quad\text{ on }\{h > 0\},
\]
for some $c > 0$ and $\ti{\gamma} > 1-\delta$.
\end{itemize}
Then for any $0 < \gamma < \min\{\ti{\gamma} - 1 + \delta, \tfrac{\delta}{2}\}$ and any $a > 0$, there exists a $C > 0$ such that:
\[
\Delta (\theta K_p(h) - a K_{p + \gamma}(h) - C \rho) \leq 0 \quad \text{ on }\{h > 0\}.
\]
\end{proposition}
\begin{proof}

Choose $q$ such that $p < q < \min\{p+\gamma, 1\}$, and similar to before, define two convex increasing functions $\chi = K_{p+\gamma}\circ L_q$ and $\ti{\chi} = K_q\circ L_p$. By completing the square and using the quasi-$m$-subharmonicity of $K_q$:
\begin{align*}
i\p\pbar (\theta K_p(h) - a K_{p+\gamma}(h)) &=  \theta i\p\pbar K_p + i\p K_p \wedge \pbar \theta + i\p\theta\wedge\pbar K_p + K_p i\p\pbar \theta \\
&- \frac{a}{2}i\p\pbar K_{p+\gamma} - \frac{a}{2}\chi'(K_q)i\p\pbar K_q - \frac{a}{2}\chi''(K_q)\ti{\chi}'(K_p)^2i\p K_p \wedge\pbar K_p\\
&\leq_m -\frac{a}{2}i\p\pbar K_{p+\gamma}  + \theta i\p\pbar K_p+ \frac{a}{2}C_q \chi'(K_q)\omega\\
&+ h^{p-\gamma} i\p\theta \wedge \pbar\theta  + C \abs{K_p} \omega.
\end{align*}
We can replace the $\omega$'s above with $i\p\pbar \rho$, by the strict $m$-subharmonicity of $\rho$. Then, similar to Lemma \ref{p lemma}, the quasi-$m$-subharmonicity of $K_q$ gives us the lower bound $i\p\pbar K_{p+\gamma} \geq_m -C_q\frac{1-p+\gamma}{1-q} h^{p+\gamma - q}i\p\pbar \rho + \frac{p+\gamma - q}{1-q}h^{p+\gamma-1}i\p\pbar h$. By $m$-ellipticity, we have $\Delta h \geq \e h^{1-\delta}\Delta \rho$, so that:
\[
\Delta K_{p+\gamma} \geq (\e' h^{p+\gamma - \delta} - C h^{p+\gamma - q})\Delta \rho,
\]
for some $\e' > 0$. It follows then that:
\[
\Delta(\theta K_p - aK_{p+\gamma}) \leq [-\e' h^{p+\gamma - \delta} + C (h^{p-1+\ti{\gamma}} + h^{p+\gamma - q} + h^{p-\gamma} + \abs{K_p})]\Delta \rho.
\]
By our choices of $q$ and $\gamma$, this is bounded above by $C\Delta \rho$.
\end{proof}


\subsection{Proof of the Theorem}\label{Main_Theorem_Section}


Theorem \ref{Theorem 1} will now follow from:
\begin{theorem}\label{Real Theorem 1}
Suppose that $h \in \mSH(\Omega)$ satisfies the conditions stated at the start of this section. Additionally, suppose that there exists $p\leq 0$, $c > 0, \ti{\gamma} > 1- \delta,$ and an $m$-elliptic operator $\Delta$ such that:
\begin{enumerate}
\item $i\p\pbar K_p(h) \geq_m -ch^{\ti{\gamma}}K_p'(h)\omega$ on $\{h > 0\}$ and
\item $\Delta K_p(h) \leq c h^{\ti{\gamma}}K_p'(h)\Delta\rho$ on $\{h > 0\}$.
\end{enumerate}
Set $\psi := K_p(h)$. Then for any non-negative $\theta\in C^2(\ov{\Omega})$, the envelope:
\begin{equation}\label{envelope for theorem}
\psi_\theta := \sup\{ \vp\in\mSH(\Omega)\ |\ \vp\leq \min\{0, \theta\psi + C\}\text{ for some } C\geq 0\}^*
\end{equation}
satisfies:
\begin{enumerate}[(i)]
\item $\Hm(\psi_\theta) = \chi_{S_\psi}\theta^m\Hm(\psi)$,
\item $\psi_\theta \geq (\max_\Omega \theta) \psi - C$,
\item $\psi_\theta \in L^\infty_\loc(\Omega\setminus \ov{\{\theta > 0\}})$, and
\item $\displaystyle \lim_{\substack{x\rightarrow z \\ x\not\in S_\psi}} \frac{\psi_\theta}{\theta\psi} = 1$ for all $z\in S_\psi\cap \{\theta > 0\}$.
\end{enumerate}
If $p > \max\{1-\delta - \ti{\gamma}, -\frac{\delta}{2}\}$, then $\psi_\theta$ and $\theta\psi$ have the same singularity type. 
\end{theorem}
\begin{proof}
The conditions $(i)$-$(iv)$ will follow readily from the propositions above. First, condition $(ii)$ is obvious. After perhaps multiplying $\theta$ by a smooth cut-off which is $\equiv 1$ on a neighborhood of $S_\psi$, we may choose functions $\theta_j$ and constants $a_j$ as in Proposition \ref{subsolution}. The resulting subsolution, $\ul{\psi}_\theta$, will be a competitor in the envelope \eqref{envelope for theorem}, after perhaps subtracting a constant. Letting the supports of the $\theta_j$ shrink to the support of $\theta$ establishes $(iii)$.

We also see that $\psi_\theta \preccurlyeq \theta\psi$ on $\{\theta > 0\}$, since $\lim_{x\rightarrow S_\psi} \frac{\ul{\psi}_\theta}{\psi} = \theta$; Proposition \ref{polar comparison} then implies $\Hm(\psi_\theta) \leq \chi_{S_\psi}\theta^m\Hm(\psi)$.

Now, for any $a > 0$, we can produce a supersolution $\ov{\psi}_\theta$ by using Proposition \ref{supersolution}. By Corollary \ref{RS_Cor}, we have $\psi_\theta \leq \ov{\psi}_\theta$, from which the rest of the conditions follow immediately. The last statement follows from Remark \ref{bounded_remark} (for Theorem \ref{Theorem 1}, note that if $\psi\in\mathcal{E}_m(\Omega)$ with $\Hm(\psi) = 0$ on $\{\psi > -\infty\}$, then we can take $\Delta = (i\p\pbar \psi)^{m-1}\wedge\omega^{n-m}$ and $\ti{\gamma}$ arbitrarily large).
%
%
%
\end{proof}
%


\subsection{Proof of the Corollary}\label{Corollary_Section}

We now show Corollary \ref{Corollary 	1}:
\begin{corollary}\label{Real Corollary 1}
Suppose that $V\subset \Omega$ is a complex submanifold, with $\p V =\emptyset$ or $\p V\subset \p\Omega$, and complex codimension $k \geq m$. Set $h := d_{\omega, V}^2$, the square of the Riemannian distance function to $V$ induced by $\omega$, and $p := 1 - \frac{k}{m}\leq 0$. Then for any $\theta\in C^2(\ov{\Omega})$, $\theta\geq 0$, the envelope:
\[
\psi_\theta := \sup\{ \vp\in\mSH(\Omega)\ |\ \vp\leq \min\{0, \theta K_p(h) + C\}\text{ for some } C\geq 0\}^*
\]
satisfies the conclusions $(i)$-$(iv)$ in Theorem \ref{Real Theorem 1}. Moreover, if $k < \frac{3}{2}m$, then $\psi_\theta$ and $\theta K_p(h)$ have the same singularity type. 
\end{corollary}
\begin{proof}
By shrinking $\Omega$ if necessary, we can assume that $h$ is smooth on $\Omega$. Then, by computing in geodesic cylindrical coordinates, one can show that $i\p\pbar h$ is approximately equal to the Euclidean distance to a codimension $k$ hyperplane in $\C^n$ \cite{TY12}; it follows that $h$ is strictly $k$-sh (and consequently, it is strictly $m$-sh for any $m\leq k$). 

More precisely, choose coordinates along a geodesic normal to $V$ such that the first coordinate is in the radial direction, the last $n-k$ coordinates are (parallel transports of directions) tangent to $V$, and the remaining coordinates are normal to $V$. Then we have:
\[
\frac{1}{h} i\p h \wedge \pbar h = \begin{spm} 1 & 0 \\ 0 & 0 \end{spm} \quad\text{ and }\quad i\p\pbar h = \begin{spm} \mathrm{Id}_{k} & 0 \\ 0 & 0 \end{spm} + o(h^{1/2}),
\]
where we write $o(h^{1/2})$ for a term such that $\tfrac{o(h^{1/2})}{h^{1/2}}\rightarrow 0$ uniformly on compacts as $h\rightarrow 0$. It follows that $h$ satisfies the conditions listed at the start of the section with $\delta = 1$. 

We now show that conditions (1) and (2) from Theorem \ref{Real Theorem 1} hold when $\ti{\gamma}$ is arbitrarily close to $\tfrac{1}{2}$. We have:
\begin{align*}
\frac{1}{K_p(h)'} i\p\pbar K_p(h) &= i\p\pbar h + \frac{p-1}{h}i\p h\wedge\pbar h\\
&\geq \begin{spm} 1 - \tfrac{k}{m} & 0 & 0\\ 0 & \mathrm{Id}_{k-1} & 0 \\ 0 & 0 & 0 \end{spm} + o(h^{1/2}).
\end{align*}
The leading order term is $m$-sh, by direct computation, so that $i\p\pbar K_p(h) \geq_m -c h^{\frac{1}{2}-\frac{\eta}{2}}K_p'(h)\omega$ for any $\eta, c > 0$ sufficiently small. This shows (1).

To see (2), start by using Proposition \ref{subsolution} to choose $a_j, C > 0$ and $\gamma = \frac{1}{2}-\eta$ such that $\ti{K} := K_p(h) + \sum_{j=1}^\ell a_j K_{p + j\gamma}(h) + C\rho$ is $m$-sh. A short computation in cylindrical coordinates shows that, on $\{h > 0\}$:
\begin{align*}
\Hm(K_p(h)) &= K_p'(h)^{m}\left(i\p\pbar h + \frac{m(p-1)}{h}i\p h\wedge\pbar h\right)\wedge (i\p\pbar h)^{m-1}\wedge\omega^{n-m}\\
&\leq K_p'(h)^{m} o(h^{1/2})\omega^n,
\end{align*}
after throwing away the leading order term. It follows that:
\begin{align*}
i\p\pbar K_p(h)\wedge (i\p\pbar \ti{K})^{m-1}\wedge\omega^{n-m} &\leq \Hm(K_p) + C \abs{i\p\pbar K_{p+\gamma}\wedge (i\p\pbar K_p)^{m-1}\wedge\omega^{n-m}}\\
&\leq K_p'(h)^{m-1}[K_p'(h) o(h^{1/2}) + CK_{p+\gamma}'(h)]\omega^n\\
&\leq C h^{\frac{1}{2}-\eta}K_p'(h) (i\p\pbar \rho)\wedge (i\p\pbar \ti{K})^{m-1}\wedge\omega^{n-m}.
\end{align*}
Condition (2) is thus satisfied for $\Delta = (i\p\pbar \ti{K})^{m-1}\wedge\omega^{n-m}$ and $\ti{\gamma} = \frac{1}{2}-\eta$.
\end{proof}


\end{document}